\documentclass[11pt]{amsart}

\makeatletter
\usepackage{amssymb,stmaryrd,verbatim}

\usepackage{pdflscape}
\usepackage[colorlinks,citecolor=blue,urlcolor=black,linkcolor=black]{hyperref}
\usepackage{amsmath, amsfonts, amsthm}

\usepackage{cmap}
\usepackage{comment}

\newtheorem{thm}{Theorem}[section] 
\newtheorem{lemma}[thm]{Lemma}
\newtheorem{cor}[thm]{Corollary}
\newtheorem{coro}{Corollary}
\newtheorem{claim}{Claim}[thm]
\newtheorem{prop}[thm]{Proposition}
\newtheorem{fact}[thm]{Fact}

\newtheorem*{thma}{Main Theorem}

\theoremstyle{definition}
\newtheorem{definition}[thm]{Definition}
\newtheorem{conv}[thm]{Convention} 

\theoremstyle{remark}

\newcommand\br{\blacktriangleright}
\newcommand\s{\subseteq}
\newcommand\sq{\sqsubseteq}

\renewcommand{\restriction}{\mathbin\upharpoonright}    

\DeclareMathOperator{\ch}{CH}

\DeclareMathOperator{\cf}{cf}
\DeclareMathOperator{\otp}{otp}
\DeclareMathOperator{\dom}{dom}
\DeclareMathOperator{\rng}{Im}
\DeclareMathOperator{\zfc}{ZFC}

\DeclareMathOperator{\lex}{lex}

\DeclareMathOperator{\acc}{acc}
\DeclareMathOperator{\nacc}{nacc}

\DeclareMathOperator{\p}{P}
\DeclareMathOperator{\h}{ht}

\newcommand{\RNum}[1]{\uppercase\expandafter{\romannumeral #1\relax}}

\begin{document}

	\date{\today}
	\title[More minimal non-$\sigma$-scattered linear orders]{More minimal non-$\sigma$-scattered linear orders}
	
	\author{Roy Shalev}
	\address{Department of Mathematics, Bar-Ilan University, Ramat-Gan 52900, Israel.}
	\email{royshalev2@gmail.com}
	\urladdr{https://roy-shalev.github.io/}

	\subjclass[2010]{Primary 03E05. Secondary 03E35}

	\maketitle
	\begin{abstract}
	Assuming an instance of the Brodsky-Rinot proxy principle holding at a regular uncountable cardinal $\kappa$, 
we construct $2^\kappa$-many pairwise non-embeddable minimal non-$\sigma$-scattered linear orders of size $\kappa$.
In particular, in G\"odel's constructible universe $L$, these linear orders exist for any regular uncountable cardinal $\kappa$ that is not weakly compact.
This extends a recent result of Cummings, Eisworth and Moore 
that takes care of all the successor cardinals of $L$.

At the level of $\aleph_1$,
their work answered an old question of Baumgartner by constructing from $\diamondsuit$ a minimal Aronszajn line that is not Souslin.
Our use of the proxy principle yields the same conclusion from a weaker assumption which holds 
for instance in the generic extension after adding a single Cohen real to a model of $\ch$.			
	\end{abstract}	
	\section{Introduction}

Fra{\"i}ss\'e conjectured in \cite{MR0028912} that the class of countable linear orders are \emph{well quasi-ordered} by embeddability, meaning that given an infinite sequence $\langle {(L_i,<_{L_i})}\mid i<\omega \rangle$ of elements in the class, there exist $i<j<\omega$ such that $(L_i,<_{L_i})$ is embeddable into $(L_j,<_{L_j})$.
In \cite{FraisseConj}, Laver proved the stronger result, that the larger class of $\sigma$-scattered linear orders is well quasi-ordered by embeddability.
Recall that a linear order $(L,{<_L})$ is \emph{scattered} if the rational line $(\mathbb Q,{<})$ does not embed into it; $(L,{<_L})$ is \emph{$\sigma$-scattered} if it is a union of countably many scattered suborders.
In his paper, Laver asks what is the behavior of the class $\mathfrak M$ of non-$\sigma$-scattered linear orders under embeddability.
We will focus on two specific characteristics:
\begin{itemize} \item The existence of a minimal element in the class; to be precise, we say that $(L,{<_L})\in\mathfrak M$ is \emph{minimal} if for every $X\subseteq L$ such that $(X,{<_L})\in \mathfrak M$, there exists an embedding from $L$ to $X$, i.e., there exists an order preserving function $f:L\rightarrow X$.
	\item The maximal cardinality of an anti-chain in the class, i.e. a family of elements which are incomparable under the embedability order.
\end{itemize}

There are two well-known different types of non-$\sigma$-scattered linear orders: \emph{real types} which are uncountable linear orders isomorphic to a subset of the real line; and  \emph{Aronszajn lines} which are linear orders of size $\aleph_1$ with the property that they do not contain uncountable suborders which are either separable or scattered.
To verify that the two types are indeed not $\sigma$-scattered one can use the following theorem of Hausdorff \cite{MR1511478}; If $\kappa$ is a regular cardinal and $(L,<_L)$ is a scattered linear order of cardinality $\kappa$, then either $(\kappa,\in)$ or the reversed order $(\kappa,\in^*)$ embeds into $(L,<_L)$.
A literature review about real types may be found in \cite{MR0661296} and \cite{CEM23}; here, we will focus on Aronszajn lines.

An important sub-class of the Aronszajn lines are the \emph{Countryman lines}; A Countryman line is an uncountable linear order $(C,<_C)$ such that its square with the coordinate-wise ordering is the union of countably many chains.
In \cite{Sh:50}, Shelah constructed the first Countryman line in $\zfc$.
Later in \cite{MR908147} Todor\v{c}evi\'{c} obtained more constructions using the method of \emph{walks on ordinals}.
It is a well-known result (see \cite[Theorem~2.1.12]{Walks_on_ordinals_book} and the discussion before Theorem~1.6 in \cite{Moore_may_be_the_only_minimal_uncounable_lin_orders}) that using a technique similar to \cite{Moore_Universal_Aronszajn_line} one can show that under \emph{Martin's Axiom} every $\aleph_1$-dense non-stationary Countryman line is minimal with respect to the class of uncountable linear orders.
Recall that a Souslin line is a non-separable linear order which has the countable chain condition under the order topology.
It is folklore fact that Souslin lines are not Countryman lines.

It was claimed by Baumgartner \cite{MR0661296} that assuming $\diamondsuit^+$,
 there exists a Souslin line which is minimal with respect to the class of uncountable linear orders (an error in Baumgartner's proof was found and fixed by D. Soukup \cite[Section~4]{Soukup_strongly_surjective}).
Baumgartner asked whether in fact $\diamondsuit$ suffices to construct a minimal Aronszajn line which was not Souslin.

In a recent paper by Cummings, Eisworth and Moore \cite{CEM23} a positive answer was given, by proving that $\diamondsuit$ implies there exists a Countryman line which is minimal with respect to the class of uncountable linear orders. 
The authors also gave the first example of a minimal non-$\sigma$-scattered linear order of size greater than $\aleph_1$.
In fact, in G\"odel's constructible universe for every infinite cardinal $\lambda$,
they constructed a minimal element of the class $\mathfrak M_{\lambda^+}$ of non-$\sigma$-scattered linear orders of size $\lambda^+$.
The construction goes through a reduction of the problem to the existence of a $\lambda^+$-Aronszajn tree with particular properties,
and then such a tree is constructed using the square with built-in diamond principle.

Next we present the contribution of this paper to the subject.
At the end of the paper \cite{CEM23}, the authors write that presumably, in the constructible universe a similar construction can be done for any regular uncountable cardinal $\kappa$ that is not weakly compact.
In this paper, the details of such a construction are given. First, the problem of the existence of a minimal element of $\mathfrak M_\kappa$ 
is again reduced to the existence of a certain $\kappa$-Aronszajn tree; this step is almost identical to that of the successor case, and we shall point out the adjustments made.
Second, instead of square with built-in diamond, the construction here will be using the Brodsky-Rinot proxy principle $\p(\ldots)$ from \cite{paper22},
as the proxy principle was designed to enable a uniform construction that simultaneously applies to all type of regular cardinals (be it $\aleph_1$, a successor of a regular, a successor of a singular or an inaccessible that is not weakly compact).
Proxy-based constructions of trees usually follow the so-called \emph{microscopic approach} as a procedure involving a transfinite application of local \emph{actions} to control global features of the
outcome tree. 
In our construction, in addition to employing actions needed to ensure that the lexicographic ordering of subtrees of the outcome tree be minimal non-$\sigma$-scattered linear orders,
we shall adapt an action from \cite{paper62} for producing many distinct trees in a way that will produce many distinct minimal orders.

\begin{thma} Assume $\kappa$ is a regular uncountable cardinal and $\p_\xi(\kappa,2,{\sq},\allowbreak\kappa)$ holds for some ordinal $\xi\le\kappa$.
	Then the class $\mathfrak M_\kappa$ of non-$\sigma$-scattered linear orders of size $\kappa$
	has $2^\kappa$-many pairwise non-embeddable minimal elements with respect to being non-$\sigma$-scattered. If $\xi<\kappa$, then the elements are all $\kappa$-Countryman lines.
\end{thma}

The definition of $\p_\xi(\kappa,2,{\sq},\kappa)$ together with a list of sufficient conditions for it to hold
may be found at the beginning of Section~\ref{sectree}.
Thus we get the following two corollaries, the first answers the question raised in \cite{CEM23} and the second regards the Baumgartner question.

\begin{coro}
	Assume $V=L$.
	For every uncountable regular non weakly compact cardinal $\kappa$, the class $\mathfrak M_\kappa$ of non-$\sigma$-scattered linear orders of size $\kappa$
	has $2^\kappa$-many pairwise non-embeddable minimal elements with respect to being non-$\sigma$-scattered. 
\end{coro}
\begin{coro}
	Assuming $\p^\bullet_\omega(\omega_1,2,{\sq},\omega_1)$ --- a consequence of $\diamondsuit$ which also holds in the generic extension after adding a single Cohen real to a model of $\ch$ ---
	there exists a family of $2^{\aleph_1}$ many Countryman lines each one is minimal with respect to being non-$\sigma$-scattered and every two members of the family are not embedded into each other.
\end{coro}

\subsection{Organization of this paper}
In Section $2$, we provide a few preliminaries on trees.

In Section $3$, we present the reduction of the problem of the existence of a minimal linear order with respect to being non-$\sigma$-scattered to the existence of a $\kappa$-Aronszajn tree with specific features.

In Section $4$, we give a construction of the tree with the sought-after features.

In Section $5$, we show how to construct $2^\kappa$-many linear orders each of size $\kappa$ and minimal with respect to being non-$\sigma$-scattered such that any two of them are not embeddable to one another.

\section{Preliminaries}
Throughout this paper, $\kappa$ denotes a regular uncountable cardinal. 

$H_\kappa$ denotes the collection of all sets of hereditary cardinality less than $\kappa$.
For a set of ordinals, $C$,
write $\acc(C) := \{\alpha\in C\mid \sup (C\cap\alpha) = \alpha>0 \}$ and $\nacc(C) := C \setminus \acc(C)$.
Given an ordinal $\alpha$, we denote by $\alpha^*$ the reversed linear order of $(\alpha,\in)$.
Given an ordinal $\alpha$ and a set $A$, let ${}^{<\alpha}A$ be the set of all functions with domain smaller than $\alpha$ and image subset of $A$, simlarly we define ${}^{\alpha}A$ as the set of all function with domain $\alpha$ and image subset of $A$.
For pair of regular cardinals $\theta<\kappa$, denote $E^\kappa_\theta:=\{\alpha<\kappa\mid \cf(\alpha)=\theta\}$.
For all $\alpha$ ordinal, $t:\alpha\rightarrow \omega$, and $i<\omega$, we denote by $t{}^\frown \langle i\rangle$ the unique function $t'$ extending $t$ and satisfying $\dom(t')=\alpha+1$ and $t'(\alpha)=i$.

\subsection{Extended background on Aronszajn lines}
Assuming the \emph{Proper Forcing Axiom}, the behavior of the class of Aronszajn lines is well classified by the the following theorems:
In \cite{Moore_5_basis}, Moore showed that every Aronszajn line has a Countryman suborder and in \cite{Moore_Universal_Aronszajn_line} he proved there exists a universal Aronszajn line, building on that, Martinez Ranero \cite{MR2822417} showed that the Aronszajn lines are well quasi-ordered by embeddability.

Regarding the question of minimal non-$\sigma$-scattered linear orders which are neither a real type nor Aronszajn line, Lamei Ramandi showed two constructions of a positive answer consistent with $ \diamondsuit$, the first \cite{MR4045989} a dense suborder of a Kurepa line of cardinality $\aleph_1$ and the second \cite{MR4383009} with the property that every
uncountable suborder contains a copy of $\omega_1$. 

Next, we move on to survey negative solutions to the minimality question.
In \cite{Moore_may_be_the_only_minimal_uncounable_lin_orders}, Moore showed it is consistent (with CH) that $\omega_1$ and $\omega_1^*$ are the only minimal uncountable linear orders.
D.~Soukup adapted this argument and showed in \cite{MR4013972} that the existence of a Souslin line does not imply the existence of a minimal Aronszajn line.

The strategy in \cite{Moore_may_be_the_only_minimal_uncounable_lin_orders}
was combined with the analysis of \cite{MR2534177} to yield the following result by Lamei Ramandi and Moore \cite{MR3934853}: If there is a supercompact cardinal,
then there is a forcing extension in which $\ch$ holds and
there are no minimal non-$\sigma$-scattered linear orders.

\subsection{Trees and linear orders}
Recall that a poset $\mathbf T=(T,<_T)$ is a $\kappa$-tree iff all of the following hold:
\begin{enumerate}
	\item For every $x\in T$, the set $x_{\downarrow}:=\{y\in T\mid y<_T x\}$ is well-ordered by $<_T$. 
	Hereafter, write $\h(x):=\otp(x_\downarrow ,<_T)$;
	\item For every $\alpha<\kappa$, the set $T_\alpha:=\{x\in T\mid \h(x)=\alpha\}$ is nonempty and has size less than $\kappa$, and the set $T_\kappa$ is empty.
\end{enumerate}
A subset $B\subseteq T$ is a branch iff $(B,<_T)$ is linearly ordered.
The height of a subset $A\subseteq T$, denoted by $\h(A)$ is the least $\alpha$ ordinal such that $T_\alpha\cap A=\emptyset$.

For a subset $X\subseteq T$, let $X_\downarrow:=\{z\in T\mid \exists x\in X[z\leq_T x]\}$.
A subset $A\subseteq T$ is a \emph{subtree} of $T$ if and $A_\downarrow =A$ and $\h(A)=\h(T)$.
Given $x\in T$, let $x^\uparrow:=\{t\in T\mid x\leq_T t\}$.
For a subset $A$ of ordinals, let $T\restriction A:=\{t\in T\mid \h(t)\in A\}$.

A \emph{$\kappa$-Aronszajn tree} is a $\kappa$-tree with no $\kappa$-sized branches, if it furthermore has no $\kappa$-sized antichains it is called a \emph{$\kappa$-Souslin tree}.
A tree $\mathbf T$ is said to be \emph{Hausdorff} iff for every limit ordinal $\alpha$ and all $x,y\in T_\alpha$, if $x_\downarrow =y_\downarrow$, then $x=y$.
A $\kappa$-tree is said to be \emph{slim} if for every infinite $\alpha<\kappa$ we have $|T_\alpha|\leq |\alpha|$.

A \emph{streamlined $\kappa$-tree} \cite{paper23} is a subset $T\subseteq {}^{<\kappa}H_\kappa$ such that the following two conditions are satisfied:
\begin{enumerate}
	\item $T$ is a downward-closed, i.e. for every $t\in T$, $\{t\restriction \alpha\mid \alpha<\kappa\}\subseteq T$;
	\item for every $\alpha<\kappa$, the set $T_\alpha:=T\cap {}^{\alpha}H_\kappa$ is nonempty and has size $<\kappa$.
\end{enumerate}

We identify a streamlined $\kappa$-tree $T$ with the poset $\mathbf T=(T,{\s})$.
Note that streamlined trees are Hausdorff.

All trees discussed in this paper will be Hausdorff and all trees constructed will be streamlined.

The \emph{lexicographic ordering} of a tree $T\subseteq {}^{<\kappa}\omega$ is a linear order $(T,{<_{\lex}})$ defined as follows:  For $s,t\in T$, we let $s<_{\lex} t$ if and only if $ s\subseteq t\text{ or } s(\Delta)<t(\Delta)$  where $\Delta:=\min\{\xi<\min\{\dom(s),\dom(t)\}\mid s(\xi)\neq t(\xi)\}$.
 	
A linear order $(L,<_L)$ is a \emph{$\kappa$-Aronszajn line} if both $\kappa$ and $\kappa^*$ are not embedded into the line and every subset of size $\kappa$ of $L$ has a density $\kappa$.

By a theorem of Hausdorff \cite{MR1511478} (see also \cite{CEM23} and \cite{MR0662564}), $\kappa$-Aronszajn lines are non-$\sigma$-scattered.
One can check that a generalization of a folklore argument (\cite[Theorem~5.1]{MR0776625}) proves that if $(T,\subseteq)$ is a $\kappa$-Aronszajn tree, then $(T,{<_{\lex}})$ is a $\kappa$-Aronszajn line.
 
We will construct a streamlined $\kappa$-Aronszajn tree $T\subseteq {}^{<\kappa} \omega$  which is not $\kappa$-Souslin such that for every antichain $X\subseteq T$ of size $\kappa$ the corresponding line $(X,{<_{\lex}})$ will be a minimal non-$\sigma$ scattered linear order, i.e. for $Y\subseteq X$ such that $(X,{<_{\lex}})$ is non-$\sigma$ scattered, then $(X,{<_{\lex}})$ embeds into $(Y,{<_{\lex}})$.

We must ensure that every subset of $T$ of size less than $\kappa$ ordered lexicographically will be $\sigma$-scattered.
To achieve this, the authors in \cite{CEM23} constructed the tree such that it is closed under the following kind of modifications.

\subsection{$\varrho$-modifications}
In this subsection, we collect some key concepts from \cite{CEM23}.
A \emph{$\varrho$-modification} is function $\eta:x\rightarrow \mathbb Z$ where $x$ is a finite subset of $\nacc(\kappa)$ and $0\in x$.
For a $\varrho$-modification $\eta$, we let $\eta^-:\dom(\eta)\rightarrow\mathbb Z$ denote the $\varrho$-modification satisfying $\eta^-(\beta)=-\eta(\beta)$ for each $\beta\in\dom(\eta)$.
We denote by $\varrho$ the set of all $\varrho$-modifications.
\begin{definition}
For two elements $\eta,\tau$ of $H_\kappa$, we define $\eta*\tau$ to be the emptyset, unless all of the following hold:

\begin{itemize}
\item $\eta\in\varrho$,
\item	$\tau\in{}^{<\kappa}\omega$,  
\item the map $f:\dom(\tau)\rightarrow \mathbb Z$ defined by stipulating 	
$$f(\beta):=\eta(\beta^-)+\tau(\beta)$$
where $\beta^-:=\max(\dom(\eta)\cap (\beta+1)))$, is in fact a function to $\omega$,
\end{itemize}
in which case, we let $\eta*\tau$ be the above function $f$.
\end{definition}
Note that $\eta:\{0\}\rightarrow\{0\}$ is an element of $\varrho$ satisfying $\eta*\tau=\tau$ for every $\tau\in{}^{<\kappa}\omega$.
Given two non-empty function $s,t$, we say that $s$ is a $\varrho$-modifcation of $t$ if there exists some $\eta\in \varrho$ such that $s=\eta*t$.

\begin{definition}
A streamlined $\kappa$-tree $T\s{}^{<\kappa}\omega$ is said to be:
\begin{itemize}
\item \emph{$\varrho$-coherent} if for all $\alpha<\kappa$ and $t,s\in T_\alpha$, there exists some $\eta\in \varrho$ such that $\eta*t=s$;
\item \emph{$\varrho$-uniform} if for every $t\in T$, it is the case that $\eta*t$ is in $T$
for every $\eta\in\varrho$;
\item \emph{uniformly $\varrho$-coherent} if it is $\varrho$-coherent and $\varrho$-uniform.
\end{itemize}
\end{definition}

\begin{definition}
	Suppose that $T \subseteq {}^{<\kappa}\omega$ is uniformly $\varrho$-coherent, $i<\omega$ and $s$ and $t$ are in $T$.
	$t$ is an  {\em $i$-extension} of $s$, written $s\subseteq_i t$,  if  $s \subseteq t$ and $t(\xi)\geq i$ whenever
	$\dom(s) \leq \xi < \dom(t)$.
\end{definition}

\begin{definition}[{\cite[Definition~3.8]{CEM23}}]
	Suppose that $T \subseteq {}^{<\kappa}\omega$ is uniformly $\varrho$-coherent.
	\begin{enumerate}
		\item 
		The {\em cone of $T$ determined by $s$}, denoted $T_{[s]}$, is defined as usual by
		\[
		T_{[s]}:=\{t\in T\mid t\subseteq s\text{ or }s\subseteq t\}.
		\]
		\item 
		The {\em frozen cone of $T$ determined by $s$ and $i$}, denoted $T_{[s,i]}$, is defined by
		\[
		T_{[s, i]}:=\{t\in T\mid t\subseteq s\text{ or }s \subseteq_i t\}.
		\]
	
	\end{enumerate}
\end{definition}

Note that $T_{[s]}=T_{[s,0]}$ and $\langle T_{[s,i]}\mid i<\omega\rangle$ is a $\subseteq$-decreasing with intersection $\{t\in T\mid t\subseteq s\}$.
Assuming $T\subseteq {}^{<\kappa}\omega$ is uniformly $\varrho$-coherent, frozen cones of $T$ of the form $T_{[s,i]} $ where $\dom(s)\in \nacc(\kappa)$ are also subtrees of $T$, i.e. they are of size $\kappa$ and downward closed.

\begin{lemma}\label{T_s embedding to T_{[s,i]}}
	Assume $T\subseteq {}^{<\kappa}\omega$ is a $\kappa$-Aronszajn tree which is uniformly $\varrho$-coherent.
	For all $s\in T\restriction \nacc(\kappa)$ and $i<\omega$, there exists a map $\varphi:T_{[s]}\rightarrow T_{[s,i]}$ which is $<_{\lex}$ preserving and preserving incompatibility in the tree.
\end{lemma}

\begin{proof}
	Let $\eta:\{0,\dom(s)\}\rightarrow \mathbb Z$ be the map defined by $\eta(0)=0$ and $\eta(\dom(s))=i$, note that $\eta\in \varrho$.
	We define $\varphi:T_{[s]}\rightarrow T_{[s,i]}$ by letting $\varphi(t):=\eta*t$.

	As $T$ is $\varrho$-uniform we get that $\varphi(t)\in T$ whenever $t\in T$.
	
	Fix $t\in T$ such that $s\subseteq t$, we claim that $\varphi(t)\supseteq_i s$.
	Note that $s\subseteq \varphi(t)$ and the for every $\dom(s)\leq \xi<\dom(t)$, $\varphi(t)(\xi)\geq i$ as sought, hence $\varphi(t)\in T_{[s,i]}$.
	
	It is clear that $\varphi$ is $<_{\lex}$ preserving and preserving incompatibility in $T$.	
\end{proof}

\begin{lemma}\label{Lemma - req. for Laver}
	Suppose that $T\subseteq {}^{<\kappa}\omega$ is a $\varrho$-uniform $\kappa$-Aronszajn tree and $s\in T$.
	Let $\lambda:=|T_{\leq \dom(s)}| $ and $\delta:=\dom(s)+\lambda +1$. Then every interval of $(T_{[s]}\cap T_\delta,{<_{\lex}})$ contains a copy of $\zeta^*$ and $\zeta$, for every $\zeta<\lambda^+$.
\end{lemma}\begin{proof}
We prove by induction on $\zeta\in[\lambda,\lambda^+)$ that every interval in $(T_{[s]}\cap T_\delta,{<_{\lex}})$ contain a copy of $\zeta$ and $\zeta^*$.
Suppose $I$ is a non-empty interval, i.e. there exists $s_0,s_1\in T_{[s]}\cap T_\delta$ such that $I=\{t\in T_{[s]}\cap T_\delta\mid s_0<_{\lex} t<_{\lex}s_1\}$.

$\br$ Assume $\zeta=\lambda$.
Let $\gamma:=\Delta(s_0,s_1)$ and for every $\gamma<\alpha<\delta$ define $\eta_\alpha:\{0,\alpha+1\}\rightarrow \omega$ to be a $\varrho$-modification where $\eta_\alpha(0)=0$ and $\eta(\alpha+1)=1$.
Note that given $\gamma<\alpha<\beta<\lambda$ we have $s_0<_{\lex} \eta_\beta*s_0<_{\lex }\eta_\alpha*s_0 <_{\lex} s_1$.

Similarly for every $\gamma<\alpha<\delta$ define $\nu_\alpha:\{0,\gamma+1,\alpha+1\}\rightarrow \omega$ to be a $\varrho$-modification where $\nu_\alpha(0)=0$, $\nu_\alpha(\gamma+1)=1$ and $\nu_\alpha(\alpha+1)= 0$;
Note that given $\gamma<\alpha<\beta<\lambda$ we have $s_0<_{\lex} \nu_\alpha*s_0<_{\lex }\nu_\beta*s_0 <_{\lex} s_1$.
As $T$ is $\varrho$-uniform, we know that $\{\nu_\alpha*s_0,\eta_\alpha*s_0\mid \gamma<\alpha<\lambda\}\subseteq T$.
Hence, $I$ contains a copy of $\zeta$ and $\zeta^*$.

$\br$ Assume $\zeta$ is a limit ordinal such that $\lambda<\zeta<\lambda^+$.
Fix an increasing continuous sequence $\langle \zeta_i\mid i<\cf(\zeta)\rangle$ of ordinals below $\zeta$, such that $\sup\{\zeta_i\mid i<\cf(\zeta)\}=\zeta$.
For each $i<\cf(\zeta)$, let $\alpha_{i+1}:=\zeta_{i+1}\setminus \zeta_i$.
Fix an $<_{\lex}$-increasing sequence $\langle t_i\mid i<\cf(\zeta)\rangle$ and a $<_{\lex}$-decreasing sequence $\langle d_i\mid i<\cf(\zeta)\rangle$ both subsets of the interval $I$.
For each $i<\cf(\zeta)$, using the recursion assumption, fix $A_i$ a subset of the interval $\{r\in I\mid t_{i}<_{\lex} r<_{\lex} t_{i+1}\}$ order-type $\alpha_{i+1}$ and $B_i$ a subset of the interval $\{r\in I\mid d_{i+1}<_{\lex} r<_{\lex} d_i\}$ of order-type $\alpha_{i+1}^*$.
Note that $A:=\bigcup_{i<\cf(\zeta)} A_i$ is a subset of $I$ of order-type $\zeta$ and $B:=\bigcup_{i<\cf(\zeta)} B_i$ is a subset of $I$ of order-type $\zeta^*$ as sought.

$\br$ Assume $\zeta$ is a successor ordinal such that $\lambda<\nu<\lambda^+$ and $\zeta=\nu+1$.
Fix by the recursion assumption some $t\in I$ and a subset $A\subset \{r\in I\mid s_0<_{\lex} r<_{\lex} t\}$ of order-type $\nu$, note that $A\cup\{t\}$ is a subset of $I$ of order-type $\zeta$ as sought.
The other case can be done similarly.
\end{proof}

\section{A reduction (finding the right tree)}
In this section we present the details of the reduction established in \cite{CEM23},
 focusing on the changes that need to be made for the general case.
\begin{fact}[Galvin,  {\cite[Thereom~3.3~$\&$~Corollary~3.4]{FraisseConj}}]\label{Theorem - Laver}
	Assume $(L,{<_L})$ and $(K,{<_K})$ are linear orders and $\lambda$ is a cardinal such that:
	\begin{enumerate}
		\item $(L,{<_L})$ and $(K,{<_K})$ are $\sigma$-scattered;
		\item $(\lambda^+)^*$ and $\lambda^+$ do not embed into $(L,{<_L})$ and $(K,{<_K})$;
		\item For an interval $(x,y)$ of $(L,{<_L})$ and every $\zeta<\lambda^+$, we have that $\zeta^*$ and $\zeta$ embeds into the interval $(x,y)$.
	\end{enumerate}
	Then $(K,{<_K})$ is embeddable into $(L,{<_L})$.	
\end{fact}

\begin{prop}\label{Proposition - the reduction}
	Assuming $T\s{}^{<\kappa}\omega$ is a streamlined $\kappa$-Aronszajn tree that is uniformly $\varrho$-coherent, 
	not $\kappa$-Souslin tree and every subtree of $T$ contains some frozen cone $T_{[s,i]}$ where $s\in T\restriction \nacc(\kappa)$ and $i<\omega$.
	Then for every antichain  $X\subseteq T$ of size $\kappa$, $(X,{<_{\lex}})$ is a minimal linear order with respect to being non-$\sigma$-scattered.
\end{prop}

\begin{proof}
	The proof is a direct consequence of the following three claims.

	\begin{claim}[{\cite[Proposition~4.5]{CEM23}}] \label{Claim - mod_sigma-scat}
If $X \subseteq T$ has cardinality less than $\kappa$, then $(X,<_{\lex})$ is $\sigma$-scattered.\qed
	\end{claim}
	
	\begin{claim}\label{Claim - operation}
		For all $s\in T$ and $i<\omega$ there is an embedding $\phi$ of $T$ into $T_{[s,i]}$ that preserves the lexicographic order $<_{\lex}$ and tree incompatibility.
	\end{claim}
	\begin{proof}
		We prove the lemma in two stages.
		First, by Lemma~\ref{T_s embedding to T_{[s,i]}} for any $s\in T$ and $i<\omega$ where $\dom(s)\in \nacc(\kappa)$ there is an embedding from the (ordinary) cone $T_{[s]}$ into the frozen cone.
		
		Next, we show for any $s\in T\restriction \nacc(\kappa)$ that $T$ can be embedded into the cone $T_{[s]}$ preserving the lexicographic order and incompatibility. 
		To do this, let $\lambda = | T_{\leq \dom(s)}|$ and define
		$\delta := \dom(s)+\lambda+1$.
		
		Note that $|T_{\leq \delta}|= \lambda=|T_{[s]}\cap T_\delta|$ as the tree $T$ is uniformly $\varrho$-coherent.
		By Claim~\ref{Claim - mod_sigma-scat}, $(T_{\leq \delta},<_{\lex})$ is a $\sigma$-scattered linear order of size $\lambda$, so by Lemma~\ref{Lemma - req. for Laver} and Fact~\ref{Theorem - Laver}, there exists a $<_{\lex}$-preserving embedding
		\[
		\phi_0:T_{\leq \delta}\rightarrow T_{[s]}\cap T_\delta.
		\]
		\noindent
		Notice that $\phi_0$ trivially preserves incompatibility since $T_{[s]}\cap T_\delta$ is an antichain.  We extend $\phi_0$ to a function $\phi:T\rightarrow T_{[s]}$ by letting for $t$ of height greater than $\delta$, $\phi(t):=\phi_0(t\restriction \delta)^\smallfrown( t\restriction(\dom(t)\setminus \delta))$. Again, our assumptions imply that the range of $\phi$ is contained in $T_{[s]}$, and the function preserves both $<_{\lex}$ and incompatibility.
	\end{proof}
	
	Assume $X\subseteq T$ is an antichain of size $\kappa$.
	Note that $(X,{<_{\lex}})$ is a $\kappa$-Aronszajn line as it is a suborder of the $\kappa$-Aronszajn line $(T,{<_{\lex}})$.
	Hence $(X,{<_{\lex}})$ is a non-$\sigma$-scattered linear order, the next claim finishes the argument.
	
	\begin{claim}[{\cite[Lemma~3.6]{CEM23}}]
		For any subset $Y\subseteq X$ which is not $\sigma$-scattered, there is an embedding of $(X,<_{\lex})$ into $(Y,<_{\lex})$.
	\end{claim}

	\begin{proof}
		Assume $(Y,{<_{\lex}})$ is a non-$\sigma$-scattered suborder of $(X,{<_{\lex}})$.
		By Claim~\ref{Claim - mod_sigma-scat}, as $Y\subseteq T$ such that $(Y,{<_{\lex}})$ is a non-$\sigma$-scattered, $Y$ must be of size $\kappa$.

		Observe $Y$ is also an antichain in $T$.
		Let $S:=Y_\downarrow$, i.e. the downward closure of $Y$ in $T$.
		By the assumption, as $S\subseteq T$ is a subtree, there exists $s\in T\restriction \nacc(\kappa)$ and $i<\omega$ such that the frozen cone $T_{[s,i]}$ is a subset of $ S$.
		Hence, by Claim~\ref{Claim - operation} there is an embedding $\phi:T\rightarrow S$ that preserves both the lexicographic order $<_{\lex}$ and incompatibility with respect to $T$'s tree order.
			
		We define a function $f:X\rightarrow Y$ by letting $f(x)$ be some element of $Y$ that extends $\phi(x)$;
		this is possible by our choice of $S$.
		Given $x<_{\lex} y$ in $X$, we know that $\phi(x)$ and $\phi(y)$ must be incompatible in $S$, and 
		$\phi(x)<_{\lex} \phi(y)$.
		But this implies 
		$f(x)<_{\lex} f(y)$ 
		as well, and we are done.
	\end{proof}\qedhere
\end{proof}
\section{Constructing the tree}\label{sectree}

\begin{definition}[A special case of {\cite[Definition~1.5]{paper22}}]\label{proxy}
$\p_\xi^-(\kappa, 2, {\sq}, \kappa)$
	asserts the existence of a sequence
	$\vec C=\left<  C_\alpha \mid \alpha < \kappa \right>$
	such that:
	\begin{enumerate}
		\item for every $\alpha\in\acc(\kappa)$, $C_\alpha$ is a club in $\alpha$ of order-type $\le\xi$;
		\item for every $\alpha < \kappa$, for every $\bar\alpha \in \acc(C_\alpha)$, $C_{\bar\alpha}=C_\alpha\cap\bar\alpha$;
		\item for every sequence $\langle B_i \mid i < \kappa \rangle$ of cofinal subsets of $\kappa$,
		there are stationarily many $\alpha<\kappa$ such that
		$\sup(\nacc(C_\alpha)\cap B_i)=\alpha$ for every $i<\alpha$.
	\end{enumerate}
\end{definition}
\begin{conv} The principle 	$\p_\xi(\kappa, 2, {\sq}, \kappa)$ stands for the conjunction of 	$\p_\xi^-(\kappa, 2, {\sq}, \kappa)$
and $\diamondsuit(\kappa)$.
If the subscript $\xi$ is omitted, then $\xi=\kappa$.
\end{conv}

\begin{fact}[{\cite[Corollary~1.10]{paper22} and \cite[Theorem~C]{paper29}}]\label{11022} Any of the following implies that $\p_\xi(\kappa,2,{\sq},\kappa)$ holds:
\begin{itemize}
\item  $\diamondsuit(\aleph_1)$ holds, $\xi=\omega$ and $\kappa=\aleph_1$;
\item $V=L$, and $\kappa=\xi$ is a regular uncountable cardinal $\kappa$ that is not weakly compact;
\item $\xi$ is an infinite cardinal, $\kappa=\xi^+$, and square with built-in diamond holds at $\xi$;\footnote{By \cite[Corollary~1.10(6)]{paper22}, $\p_\xi(\xi^+,2,{\sq},\xi^+)$ is strictly weaker than square with built-in diamond holding at $\xi$.} 
\item $\xi$ is a singular cardinal, $\kappa=\xi^+=2^\xi$ and  $\square_\xi$ holds;
\item  $\lambda$ is a singular strong limit cardinal, $\kappa=\xi=\lambda^+=2^\lambda$ and $\square(\kappa)$ holds.
\end{itemize}
\end{fact}

\begin{lemma}\label{Lemma - The wanted sequence}
	 If $\p_\xi^-(\kappa, 2, {\sq}, \kappa)$ holds,
	 then there exists a sequence $\vec C=\langle C_\alpha\mid \alpha<\kappa\rangle$
	 such that:
	 \begin{enumerate}
 	\item for every $\alpha\in\acc(\kappa)$, $C_\alpha$ is a club in $\alpha$ of order-type $\le\xi$;
	 \item for every $\alpha < \kappa$, for every $\bar\alpha \in \acc(C_\alpha)$, $C_{\bar\alpha}=C_\alpha\cap\bar\alpha$;
	 \item the set $\{\alpha<\kappa\mid \otp(C_\alpha)=\omega\}$ is cofinal in $\kappa$;
	 \item there is no club $D\s\kappa$ such that $D\cap\alpha=C_\alpha$ for every $\alpha\in\acc(D)$;
	 \item for every sequence $\langle B_i \mid i < \kappa \rangle$ of cofinal subsets of $\acc(\kappa)$,
	 there are stationarily many $\alpha<\kappa$ such that
	 $\sup(\nacc(C_\alpha)\cap B_i)=\alpha$ for every $i<\alpha$.
	 \end{enumerate}
\end{lemma}
\begin{proof}
	Fix a sequence $\langle D_\alpha\mid \alpha<\kappa\rangle$ witnessing $\p^-(\kappa, 2, {\sq}, \kappa)$ holds.
	We define a sequence $\vec C=\langle C_\alpha\mid \alpha<\kappa\rangle$ as follows:
	
	$\br$ If $\alpha=0$, then let $C_\alpha:=\emptyset$.
	
	$\br$ If $\alpha=\beta+1$ for some $\beta<\kappa$, then let $C_\alpha:=\{\beta\}$.
	
	$\br$ If $\alpha=\omega\beta$ for some successor ordinal $\beta<\kappa$, then let $C_{\alpha}$ be a cofinal subset of $\alpha$ of order-type $\omega$.
	
	$\br$ Otherwise, i.e., $\alpha=\omega\beta$ for some $\beta\in \acc(\kappa)$, let $C_{\alpha}:=\{ \omega\gamma\mid \gamma\in D_\beta \}$.

	Clearly $(1)$ holds.
	To see that $(2)$ holds, fix $\alpha\in \acc(\kappa)$ and suppose $\bar \alpha\in \acc(C_\alpha)$.
	This implies that $\alpha=\omega\beta$ and $\bar\alpha=\omega\bar \beta$
	where $\beta\in \acc(\kappa)$ and $\bar \beta\in D_\beta$. 
	It is easy to check that $\bar\beta\in \acc(D_\beta)$ and so $D_{\bar\beta}=D_\beta\cap \bar\beta$, hence $C_{\bar\alpha}=C_\alpha\cap \bar\alpha$ as sought.
	
	By the definition of the sequence $\vec C$, the set $\{\alpha<\kappa\mid \otp(C_\alpha)=\omega\}$
	covers $\nacc(\acc(\kappa))$, so (3) holds.
	Also, the proof of \cite[Lemma~3.2]{paper22} makes it clear that (4) follows from (5).
	
	Finally, verifying $(5)$ holds, let us fix a sequence $\langle B_i\mid i<\kappa\rangle $ of cofinal subsets of $\acc(\kappa)$.
	Note that the map $\pi:\kappa\rightarrow \acc(\kappa)$ defined by $\pi(\alpha)=\omega\alpha$ is an order-preserving bijection.
	For each $i<\kappa$, let $A_i=\pi^{-1}[B_i]$.
	Clearly $\langle A_i\mid i<\kappa\rangle$ is a sequence of cofinal subset of $\kappa$. 
	
	Notice that $D:=\{\alpha<\kappa\mid \omega\alpha=\alpha\}$ is a club subset of $\kappa$.
	As $\vec D$ is a sequence witnessing $\p^-(\kappa, 2, {\sq}, \kappa)$, there are stationarily many $\alpha\in D$ such that $\sup(\nacc(D_\alpha)\cap A_i)=\alpha$ for every $i<\alpha$.
	Let us fix such an $\alpha$, note that $C_\alpha=\{\omega\beta\mid \beta\in D_\alpha\}	$.
	Fix $i<\alpha$ and $\epsilon<\alpha$.
	Pick some $\beta\in \nacc(D_{\alpha}\cap A_i)$ above $\epsilon$.
	Note that $\omega\beta\in \nacc(C_\alpha)\cap B_i$ and $\epsilon<\omega\beta$.
	We showed that $\sup(\nacc(C_\alpha)\cap B_i)=\alpha$ for each $i<\alpha$ as sought.
\end{proof}

\begin{fact}[{\cite[Lemma~2.2]{paper22}}] $\diamondsuit(\kappa)$ is equivalent to the principle
	$\diamondsuit(H_\kappa)$ asserting the existence of a partition $\langle R_i\mid i<\kappa\rangle$ of $\kappa$ and a sequence $\langle \Omega_\beta\mid \beta<\kappa\rangle$ of elements of $H_\kappa$ such that for all $p\in H_{\kappa^+}$, $i<\kappa$ and $\Omega\subseteq H_\kappa$, there exists an elementary submodel $\mathcal M\prec H_{\kappa^+}$ such that:
	\begin{itemize}
		\item $p\in \mathcal M$;
		\item $\mathcal M\cap \kappa\in R_i$;
		\item $\mathcal M\cap\Omega = \Omega_{\mathcal M\cap \kappa}$.
	\end{itemize}
\end{fact}

\begin{thm}\label{Proposition - the construction} Suppose $\p(\kappa,2,{\sq},\kappa)$ holds.
	Then there exists a uniformly $\varrho$-coherent streamlined $\kappa$-Aronszajn tree $T\s{}^{<\kappa}\omega$ that is not $\kappa$-Souslin and such that every subtree of $T$
contains some frozen cone $T_{[s,i]}$ where $s\in T\restriction \nacc(\kappa)$ and $i<\omega$.
\end{thm}
\begin{proof} We follow the proof of Proposition~2.5 from \cite{paper22}.
	Let $\vec C=\langle C_\alpha\mid \alpha<\kappa\rangle$ be a sequence as in Lemma~\ref{Lemma - The wanted sequence} (using $\xi=\kappa$),
	without loss of generality, we may assume that $0\in C_\alpha$ for all nonzero $\alpha<\kappa$.
	Let $\langle R_i \mid i<\kappa\rangle$ and $\langle \Omega_\beta\mid\beta<\kappa\rangle$ together witness $\diamondsuit(H_\kappa)$.
	Let $\pi:\kappa\rightarrow\kappa$ be such that $\alpha\in R_{\pi(\alpha)}$ for all $\alpha<\kappa$.
	By $\diamondsuit(\kappa)$, we have $\left|H_\kappa\right| =\kappa$,
	thus let $\lhd$ be some well-ordering of $H_\kappa$ of order-type $\kappa$,
	and let $\psi:\kappa\leftrightarrow H_\kappa$ witness the isomorphism $(\kappa,\in)\cong(H_\kappa,\lhd)$.
	Put $\varphi:=\psi\circ\pi$.
	
	We shall now recursively construct a sequence $\langle T_\alpha\mid \alpha<\kappa\rangle$ of levels
	whose union will ultimately be the desired tree $T$.
	
	Let $T_0:=\{\emptyset\}$, and for all $\alpha<\kappa$, let $T_{\alpha+2}:=\{ t{}^\smallfrown\langle n\rangle\mid t\in T_{\alpha+1}, n\in \omega\}$.
	For $\alpha\in \acc(\kappa)$ and $t\in T_\alpha$, let $I_{t}:=\{n<\omega\mid \sup(t^{-1}(n))=\alpha\}$.
 	Set $T_{\alpha+1}:=\{ t{}^\smallfrown\langle\min(I_t)\rangle \mid t\in T_{\alpha}, I_t\neq \emptyset\}$.
	
	Next, suppose that $\alpha\in\acc(\kappa)$ and that $\langle T_\beta\mid \beta<\alpha\rangle$ has already been defined.
	We shall first identify some node $\mathbf b^\alpha:\alpha\rightarrow \omega$,
	and then let $$T_\alpha := \{ \eta*\mathbf b^\alpha\mid \eta\in\varrho, \dom(\eta)\s\alpha\}\setminus\{\emptyset\}.$$
	
	During the construction of $\mathbf b^\alpha$ we assure that:
	\begin{enumerate}
		\item $C_\alpha =(\mathbf b^\alpha)^{-1}(0)$;
		\item if $\beta\in \acc(C_\alpha)$, then $\mathbf b^\beta\subseteq \mathbf b^\alpha$.
	\end{enumerate}
	
	$\mathbf b^\alpha$ will be obtained as a limit of a sequence $b^\alpha\in\prod_{\beta\in C_\alpha}T_\beta$ that we define by recursion.
	Let $b^\alpha(0):=\emptyset$.
	Next, suppose $\beta^-<\beta$ are successive points of $C_\alpha$, and $b^\alpha(\beta^-)$ has already been defined.

	In order to decide $b^\alpha(\beta)$, we consider the following two possibilities:
	\begin{itemize}
	\item Assuming $\varphi(\beta)\in \varrho$ and the following set is non-empty: 
	$$Q^{\alpha, \beta} := \{ t\in T_\beta\cap T_{[b^\alpha(\beta^-)^\smallfrown\langle 0\rangle ,1]}\mid \exists s\in (T\restriction\beta)\setminus\Omega_{\beta}[  \varphi(\beta)*(b^\alpha(\beta^-)^\smallfrown\langle 0\rangle)\s s~\wedge \varphi(\beta)^-*s\s t]\}.$$
	Then let $t$ denote its $\lhd$-least element, and put $b^\alpha(\beta):=t$;
	
	\item Else, let $b^\alpha(\beta)$ be the $\lhd$-least element of $T_\beta\cap T_{[b^\alpha(\beta^-)^\smallfrown\langle 0\rangle,1]}$ that extends $b^\alpha(\beta^-)$.
	\end{itemize}
	
	Note that $Q^{\alpha,\beta}$ depends only on $T_\beta,\Omega_\beta,\varphi(\beta),C_\alpha\cap \beta^-$ and $b^\alpha(\beta^-)$,
	and hence for every ordinal $\gamma<\kappa$, if $C_\alpha\cap(\beta+1)=C_\gamma\cap(\beta+1)$, then $b^\alpha\restriction(\beta+1)=b^\gamma\restriction(\beta+1)$.
	It follows that for all $\beta \in \acc(C_\alpha)$ such that $b^\alpha\restriction\beta$ has already been defined,
	we may let $b^\alpha(\beta):=\bigcup\rng(b^\alpha\restriction\beta)$ and infer that $b^\alpha(\beta)=\mathbf b^\beta\in T_\beta$.
	This completes the definition of $b^\alpha$ and its limit $\mathbf b^\alpha=\bigcup\rng(b^\alpha)$.

	This completes the definition of the level $T_\alpha$.
	
	Clearly (1)--(2) holds from the construction.
	
	Having constructed all levels of the tree, we then let
	\[
	T := \bigcup_{\alpha < \kappa} T_\alpha.
	\]
	
	\begin{claim} 
		For every $\alpha<\kappa$ and every two nodes of $t,s\in T_\alpha$, there exists some $\eta\in \varrho$ such that $s=\eta*t$.
	\end{claim}
	\begin{proof} 
		First note that if $\alpha$ is finite, then the function $\eta:\alpha\rightarrow \omega$ defined by $\eta(n)=s(n)-t(n)$ for each $n<\alpha$ is a $\varrho$-modification such that $s=\eta*t$.
		Assuming $\alpha$ is infinite, let $\beta=\max(\acc(\alpha+1))$,	
		then $t\restriction \beta$ and $s\restriction \beta$ are in $T_\beta$, and by the definition of $T_\beta$, $t\restriction \beta$ is a $\varrho$-modification of $s\restriction \beta$.
		So $t$ is a $\varrho$-modification of $s$.
	\end{proof}
	
	\begin{claim}
		The tree $T\s{}^{<\kappa}\omega$ is a slim $\kappa$-tree  that is uniformly $\varrho$-coherent.
	\end{claim}
	\begin{proof}
		The fact that $T$ is a uniformly $\varrho$-coherent $\kappa$-tree follows from the construction. 
		So for each $\alpha < \kappa$,	$T_\alpha$ is a subset of $^\alpha\kappa$ of size $\le\max\{\aleph_0,|\alpha|\}$.
		Hence, $T$ is also a slim $\kappa$-tree.
	\end{proof}
	\begin{claim}
		$T$ is a $\kappa$-Aronszajn tree.
	\end{claim}
	\begin{proof}
		Assume on the contrary that there exists $d\in {}^{\kappa}\omega$ a branch through $T$.
		For every $\alpha\in \acc(\kappa)$, let $t_\alpha:=d\restriction \alpha$.
		Since $t_\alpha\in T_\alpha$, fix some $\eta_\alpha\in\varrho$ with $\dom(\eta_\alpha)\subseteq \alpha$ such that $t=\eta_\alpha* \mathbf b^\alpha$.
		By Fodor's lemma, there exist some stationary subset $R\subseteq\kappa$, an $\eta\in\varrho$ and $t\in T$ such that for every $\alpha\in R$:
		\begin{itemize}
		\item $\eta_\alpha= \eta$, and
		\item $\mathbf b^\alpha\restriction(\max(\dom(\eta))+1)=t$.
		\end{itemize}
		
		For every $\alpha<\beta$ from $R$, as $\eta*\mathbf b^\alpha=t_\alpha\subset t_\beta=\eta *\mathbf  b^\beta$, we get that 
		$\mathbf b^\alpha \subset \mathbf b^\beta$. Recalling that $C_\alpha =(\mathbf b^\alpha)^{-1}(0)$,
		this means that $\langle C_\alpha\mid\alpha\in R\rangle$ is thread through $\vec C$,
		so $D:=\bigcup_{\alpha\in R}C_\alpha$ contradicts Clause~(4) of Lemma~\ref{Lemma - The wanted sequence}.
	\end{proof}
	
	\begin{claim}\label{Claim - not kappa souslin}
		$T$ is not a $\kappa$-Souslin tree.
	\end{claim}
	\begin{proof} 	Recall that $R:=\{\alpha\in\acc(\kappa)\mid\otp(C_\alpha)=\omega\}$ is cofinal in $\kappa$.
		Let $A:=\{\mathbf b^\alpha\mid \alpha\in R\}$; we claim $A$ is an antichain of $T$ of size $\kappa$.
		Fix $\alpha<\beta$ in $R$, we show that $\mathbf b^\alpha$ and $\mathbf b^\beta$ are incomparable in $T$.
		On the contrary, assume $\mathbf b^\alpha\subseteq\mathbf b^\beta$.
		Recall that $C_\alpha=(\mathbf b^{\alpha})^{-1}(0)$ and the only extension of $ \mathbf b^\alpha$ in $T_{\alpha+1}$ is the function $\mathbf b^\alpha {}^\smallfrown \langle 0\rangle $.
		So $\omega<\otp(C_\alpha)+1\leq \otp(C_\beta) =\omega$ which is an absurd.
	\end{proof}

	\begin{claim}
		Every subtree of $T$ contains some frozen cone $T_{[s,i]}$ where $s\in T\restriction \nacc(\kappa)$ and $i<\omega$.
	\end{claim}
	\begin{proof}
		Let $A\subseteq T$ be a subtree which does not contain some frozen cone $T_{[s,i]}$ where $s\in T\restriction \nacc(\kappa)$ and $i<\omega$.
		We aim to get a contradiction, showing that $A$ is a subset of $T\restriction \alpha$ for some $\alpha<\kappa$.
		
		We define for every $i<\kappa$, the set $A_i$ to be all $\beta\in R_i$ such that $A\cap(T\restriction\beta)= \Omega_\beta$ is a subtree of $T\restriction\beta$ which does not contain the frozen cone $T_{[s,i]}$ whenever $s\in T\restriction \nacc(\beta)$ and $i<\omega$.
		By a proof similar to that of \cite[Claim~2.3.2]{paper22}, for every $i<\kappa$, the set $A_i$ is stationary.
		For every $i<\kappa$, we define $B_i:=A_i\cap \acc(\kappa)$, note that this set is a cofinal subset of $\acc(\kappa)$.
		
		For every $\alpha<\kappa$, we let $m_\alpha$ be the set of all of $\varrho$-modifiers with domain subset of $\alpha$.
		Thus, we apply clause~$(4)$ of Lemma \ref{Lemma - The wanted sequence} to the sequence $\langle B_i\mid i<\kappa\rangle$ of cofinal subsets of $\acc(\kappa)$, and the club $D:=\{ \alpha<\kappa\mid m_\alpha\s\psi[\alpha]\}$
		to obtain
		an ordinal $\alpha\in D$ such that for all $i<\alpha$:
	$$\sup (\nacc(C_\alpha) \cap B_i) = \alpha.$$
		
		We will prove that $A \subseteq T \restriction \alpha$; fix $y\in T_\alpha$, we will show that $y\notin A$.
		By the construction of the level $T_\alpha$, $y$ is a $\varrho$-modification of $ \mathbf b^\alpha$, i.e. for some $\varrho$-modifier $\eta$ where $\dom(\eta)\subseteq\alpha$, $y=\eta*\mathbf b^\alpha$.
		As $\alpha\in D$ and $\eta \in m_\alpha$, we can fix $i<\alpha$ such that $\psi(i)=\eta$.
		
		Fix $\beta \in \nacc(C_\alpha) \cap B_i$. 
	    Clearly, $\varphi(\beta)=\psi(\pi(\beta))=\psi(i)=\eta$.
		Since $\beta \in B_i$,
		we know that $\Omega_\beta = A \cap (T \restriction \beta)$ is a subtree of $T \restriction \beta$ which does not contain the frozen cone $T_{[s,i]}$ whenever $s\in T\restriction \nacc(\kappa)$ and $i<\omega$.
		Let $N:=\max\{|n|+1\mid n\in \rng(\varphi(\beta))\}$.
		
		We claim that $Q^{\alpha, \beta}\neq\emptyset$.
		As $\Omega_\beta$ is a subtree of $T \restriction \beta$ which does not contain the frozen cone $T_{[s,i]}$ whenever $s\in T\restriction \nacc(\kappa)$ and $i<\omega$.
		we get that $T_{[\eta * \mathbf (b^\alpha(\beta^-)^\smallfrown\langle 0\rangle),N]}\not\subseteq \Omega_\beta$. 
		So we can find some $s\in T\restriction \beta$ an $N$-extension of $\eta * \mathbf (b^\alpha(\beta^-)^\smallfrown\langle 0\rangle )$ in $T\restriction \beta$ not in $\Omega_\beta$.
		This implies that $\eta^-*s$ is a well defined function in $T$.
		So $Q^{\alpha,\beta}\neq \emptyset$.
		
		Let $t:=\min(Q^{\alpha,\beta},\lhd)$ and  $\beta^-:=\sup(C_\alpha\cap\beta)$.
		So there exists some $s\in T\restriction \beta \setminus 	\Omega_{\beta}$ such that $\varphi(\beta)*(\mathbf b^\alpha(\beta^-)^\smallfrown\langle 0\rangle)\s s$ and $\varphi(\beta)^-* s\s t$.
		In particular, $\eta*(\mathbf b^\alpha(\beta)^\smallfrown\langle 0\rangle)$ extends an element not in $\Omega_\beta$.
		
		Altogether, as $A$ is downward closed, $s\notin A$ and $ s\subseteq \eta*\mathbf b^\alpha(\beta)\subseteq y$, we get that $y\notin A$.
	\end{proof}\qedhere
\end{proof}

\begin{cor}
	Assume $V=L$ and $\kappa$ is a regular uncountable cardinal that is not weakly compact.
Then there exists a minimal linear order with respect to being non-$\sigma$-scattered.
\end{cor}
\begin{proof} By Proposition~\ref{Proposition - the reduction}, Fact~\ref{11022} and Theorem~\ref{Proposition - the construction}.
\end{proof}

\subsection{Assuring $T\restriction \nacc(\kappa)$ is special}\

In this subsection we give a sufficient condition on the sequence $\vec C$ to assure that the tree constructed in the proof of Theorem~\ref{Proposition - the construction}  is such that $T\restriction \nacc(\kappa)$ is special.
It is left open whether it is consistent for an inaccessible cardinal $\kappa$ to admit such a $C$-sequence.

\begin{lemma} Suppose:
\begin{enumerate}
\item $\diamondsuit(\kappa)$ holds;
\item $\vec C=\langle C_\alpha\mid \alpha<\kappa\rangle$ is a witness for $\p^-(\kappa,2,{\sq},\kappa)$;
\item there exist a club $C\s\kappa$ and a regressive map $g: C\rightarrow \kappa$ 
satisfying that for every $\nu<\kappa$ there exists a cardinal $\mu_\nu<\kappa$ and a map $f_{\nu}:g^{-1}(\nu)\rightarrow \mu_\nu$ such that for every $i<\mu_\nu$, the set $\{ C_\alpha \mid \alpha\in (f_{\nu})^{-1}(i)\}$ is an $\sqsubseteq$-antichain.
\end{enumerate}

	Then the tree $T$ produced by Theorem~\ref{Proposition - the construction} satisfies that $T\restriction \nacc(\kappa)$ is special, i.e there exists a function $G:T\restriction \nacc(\kappa)\rightarrow \nacc(\kappa)$ such that for every $t\in T\restriction (\nacc(\kappa)\setminus 1)$, we have $G(t)<\dom(t)$ and $G^{-1}(\nu)$ is a union of ${<}\kappa$ many $<_T$-antichains for each $\nu\in \nacc(\kappa)$.
\end{lemma}

\begin{proof}
	Let us fix a bijection $\pi:[\kappa\times \mathbb Z]^{<\omega}\times \kappa \rightarrow \kappa$ and a club $D\subseteq C\cap \acc(\kappa)$ such that for every $\alpha\in D$ we have $\pi[[\alpha\times \mathbb Z]^{<\omega}\times \alpha]\subseteq \alpha$.
	Let $A:=\{\alpha+1\mid \alpha\in D\}$.
	
	Next, we define a map $G:T\restriction \nacc(\kappa)\rightarrow \nacc(\kappa)$ as follows:
	given $t\in T\restriction \nacc(\kappa)$ we split to cases:
	\begin{itemize}
		\item If $\dom(t)\in \nacc(\kappa )\cap\min(D)$, let $G(t)=0$.
		\item If $\dom(t)\in \nacc(\kappa)\setminus A$, let $G(t):=\max(D\cap \dom(t))+1$.
		Note that $G(t)<\dom(t)$ and $G(t)\in A$.
		
		\item If $\dom(t)\in A$, fix $\alpha\in D$ and $\eta_t$ a $\varrho$-modification with domain $\subseteq\alpha$ such that $t=\eta_t*(b^\alpha{}^\smallfrown \langle 0\rangle)$.
		Note that $\eta_t \in [\alpha\times \mathbb Z]^{<\omega}$ and $g(\alpha)<\alpha$.
		Let $G(t):=\pi(\eta_t,g(\alpha))+2$, clearly $0<G(t)<\alpha<\dom(t)$ and $G(t)\notin A$.
	\end{itemize}
	Fix some $\nu\in \rng(G)$, we claim that $G^{-1}(\nu)$ is a union of ${<\kappa}$ many $<_T$-antichains.
	We split to cases:
	
	$\br$ If $\nu=0$ or $\nu\in A$.
	Then every $t\in G^{-1}(\nu)$ is such that $\dom(t)\in \nacc(\kappa)\setminus A$ and $\dom(t)<\min(D\setminus \nu)$.
	As $T$ is a $\kappa$-tree, we get that $|G^{-1}(\nu)|<\kappa$, hence $G^{-1}(\nu)$ is a union of $<\kappa$ many singletons, i.e. $<_T$-antichains.
	
	$\br$ If $0\neq \nu \notin A$.
	Then every $t\in G^{-1}(\nu)$ is such that $\dom(t)\in A$, i.e. for some $\alpha_t\in D$ and $\eta_{t}$ we have $t=\eta_t *(b^\alpha{}^\frown\langle 0\rangle)$.
	As $\pi$ is a bijection, we can fix $\eta\in \varrho$ and $\beta<\kappa$ such that $\nu=\pi(\eta,\beta)+2$ and for every $t\in G^{-1}(\nu)$, we have that $g(\alpha_t)=\beta$ and $\eta_t=\eta$.
	We claim that for every $i<\mu_\beta$ the set $\{t\in G^{-1}(\nu)\mid f_{\beta}(\alpha_t)=i\} $ is $<_T$-antichain.
	Suppose on the contrary that $t,s\in G^{-1}(\nu)$, $f_{\beta}(\alpha_t)=f_{\beta}(\alpha_s)=i$ and $t\subseteq s$.
	Then as $\eta_t=\eta_s=\eta$, we get that $b^{\alpha_t}{}^\frown\langle 0\rangle\subseteq b^{\alpha_s}{}^\frown\langle 0\rangle$.
	Hence $C_{\alpha_t}\sqsubseteq C_{\alpha_s}$ which is absurd as $\alpha_t,\alpha_s\in f_{\beta}^{-1}(i)$.
\end{proof}
\section{A large antichain}

In this section we construct $2^\kappa$-many linear orders which are minimal with respect to being non-$\sigma$-scattered and every two of them do not embed into each other.

As can be seen in the surveys of \cite{MR0776625}, \cite{MR2406209}, there is a fundamental connection between the classes of Aronszajn lines and Aronszajn trees.
To give a broader context to this section, we mention some related results.
In \cite{MR0788070}, the authors proved that assuming $2^{\aleph_0}<2^{\aleph_1}$ there is a family of $2^{\aleph_1}$-many $\aleph_1$-Aronszajn trees each one is not club-embeddable into the other and that under the \emph{Proper Forcing Axiom} all $\aleph_1$-Aronszajn trees are club-isomorphic.
Later, in \cite{Lipschitz_maps_on_trees} Todor\v{c}evi\'{c} showed an analysis under the \emph{Proper Forcing Axiom} of the behavior of the class of $\aleph_1$-Aronszajn trees and the importance of the sub-class of coherent $\aleph_1$-Aronszajn trees.
Finally, we mention that in \cite[Theorem~3.3.4]{martinez2011contributions}, using the tree $T(\rho_0)$, it was shown that $2^{\aleph_0}<2^{\aleph_1}$ implies there exists a family of $2^{\aleph_1}$ pairwise incomparable $\aleph_1$-Countryman lines.

Assuming $T$ is a Hausdorff tree and $s,t\in T$, we say that $r$ is the meet of them if it is the maximal element in the tree which is below both $s$ and $t$.
The following Lemma can be extracted from the proof of \cite[Fact~2]{MR3633794}.

\begin{lemma}\label{Lemma - Aron incom}
	Suppose $\mathbf T:=(T,<_T,<_{\lex})$ is a lexicographically ordered Hausdorff $\kappa$-Aronszajn tree and $N$ is an elementary submodel such that $T\in N$ and $N\cap\kappa\in\kappa$.
	Given $s,t\in T_{N\cap \kappa}$ such that $s<_{\lex} t$, there exists some $r\in T\restriction (N\cap \kappa)$ such that $r$ is $<_T$-incomparable with both $s$ and $t$, $r$ is $<_T$-above the meet of $s$ and $t$,
	and $s<_{\lex} r <_{\lex} t$.
\end{lemma}
\begin{proof}			
	We assume on the contrary that no such $r$ can be found.
	Let $\delta:=N\cap \kappa$ and $m$ be the meet of $s$ and $t$ in the tree.
	Let $\Delta:=\h(m)+1$.
	For each $\alpha<\delta$, let $s_\alpha$ denote the unique element in $s_\downarrow \cap T_\alpha$.
	Let $a:=s_{\Delta}$.
	As $T$ is Hausdorff, $\Delta<\delta$.
	Clearly, $D:=\{\alpha<\kappa\mid |T\restriction \alpha|=|\alpha|\}$ is a club in $\kappa$ definable in $N$, which implies $\delta\in D$.
	So for every $\alpha<\delta$ we have $|T_\alpha|\leq|\delta|$, hence $T_\alpha\subseteq N$.
	This proves that $s_\alpha\in N$ for every $\alpha<\delta$.
	
	For each $\alpha<\kappa$ above $\Delta$, we define the set $X_\alpha:=\{q\in T_\alpha\mid a<_T q\}$.
	\begin{claim}
		For every $\alpha<\kappa$ above $\Delta$, the set $X_\alpha$ has a $<_{\lex}$-maximal element.
		Moreover, if $\alpha<\delta$, then $s_\alpha=\max(X_\alpha,<_{\lex})$.
	\end{claim}
	\begin{proof}
		Assume on the contrary that for some $\alpha<\kappa$ above $\Delta$ we have that the set $X_\alpha$ has no $<_{\lex}$-maximal element.
		As $N$ is an elementary submodel, we can find such an $\alpha$ in $N$, i.e $\Delta<\alpha<\delta$.
		Note that $s_\alpha\in X_\alpha\cap N$, therefore we can fix some $q\in X_\alpha$ which is $<_{\lex}$ above $s_\alpha$.
		As $a<_T q$, we get that $q$ is $<_T$-incomparable with $t$ and since $a<_{\lex} t$ we get that $q<_{\lex} t$.
		We know that $q$ is $<_T$-incomparable with $s$, as $q\neq s_\alpha$ and $q\in T_\alpha$.
		Note that $q$ is above the meet of $s$ and $t$ and $s<_{\lex} q$.
		Finally, this is contradiction to the assumption on $r$, which is absurd. \qedhere
	\end{proof}
	
	Next, for each $\alpha<\kappa$ above $\Delta$, let $z_\alpha:=\max(X_\alpha,<_{\lex})$.
	The next claim is a contradiction to the assumption $T$ is $\kappa$-Aronszajn tree.
	\begin{claim}
		The set $\{ z_\alpha\mid \Delta<\alpha<\kappa\}$ is a branch in $T$ of size $\kappa$.
	\end{claim}
	\begin{proof}
		Suppose on the contrary that $\{z_\alpha\mid \Delta<\alpha<\kappa\}$ is a not $<_T$-linearly ordered.
		As $N$ is an elementary submodel, we can find $\alpha<\beta<\delta$ above $\Delta$ such that $z_\alpha$ and $z_\beta$ are $<_T$-incomparable.
		Since $s_\alpha\in X_\alpha\cap N$ and by previous Claim, no element in $X_\alpha$ is $<_{\lex}$ above $s_\alpha$, so $z_\alpha=s_\alpha$.
		Similarly, we get that $z_\beta=s_\beta$.
		But this is absurd as $s_\alpha<_T s_\beta$.
	\end{proof}
\end{proof}

\begin{lemma}[{\cite[Proposition 2.4]{yinhe2013characterization}}]\label{Lemma - club iso}
	If $(S,<_S, <_{lS})$, $(T,<_T,<_{lT})$ are two lexicographically ordered $\kappa$-Aronszajn trees, $X$ and $Y$ are subsets of $S$ and $T$ respectively, both of size $\kappa$ and $\pi:(X,<_{lS})\rightarrow (Y,<_{lT})$ is an order isomorphism, then there exists a club $C$ such that $((X_\downarrow)\restriction C,<_S,<_{lS})$ is tree isomorphic and order isomorphic to $((Y_\downarrow)\restriction C,<_T,<_{lT})$. 
	Moreover, the witness map $f: (X_\downarrow )\restriction C\rightarrow (Y_\downarrow)\restriction C$ is such that $f$ agrees with $\pi$, i.e., $\pi"(s^\uparrow \cap X)=f(x)^\uparrow\cap Y$ for any $x\in (X_\downarrow) \restriction C$.
\end{lemma}
\begin{proof}
	Without loss of generality we may assume that $S=X_\downarrow $ and $T=Y_\downarrow$ and $\pi$ is a surjective map.
	Let us fix $\langle N_\alpha\mid \alpha<\kappa\rangle$ a continuous elementary $\in$-chain where $N_0$ contains all relevant objects, $|N_\alpha|<\kappa$ for each $\alpha<\kappa$.
	Notice the set $C:=\{\alpha<\kappa\mid \kappa\cap N_\alpha=\alpha\}$ is a club subset of $\kappa$.
	
	Note that given $t\in T$ and $s\in S$ the sets $s^{\uparrow}\cap X$ and $t^{\uparrow}\cap Y$ are convex sets in $(X,<_{lS})$ and $(Y,<_{lT})$ respectively.
	We define a map $f:S\restriction C\rightarrow T\restriction C$ as follows: for any $\alpha\in C$ and $s\in S_\alpha$, $f(s)$ is the unique $t\in T_\alpha$ such that $t^{\uparrow}\cap Y=\pi"(s^{\uparrow}\cap X)$.
	The following claim shows that the map $f$ is well-defined.
	
	\begin{claim}
		For any $\alpha\in C$ and $s\in S_\alpha$ there exists $t\in T_\alpha$ such that $t^{\uparrow}\cap Y=\pi"(s^\uparrow \cap X)$.
	\end{claim}
	\begin{proof}
		Note that $\alpha=N_\alpha\cap \kappa$ and the functions $\pi,\h_S$ and $\h_T$ are definable in $N_\alpha$.
		Fix some $x\in X$ such that $s<_S x$.
		Clearly $s,\pi(x)\notin N_\alpha$ and $\h_T(\pi(x))\geq \alpha$.
		Let $t$ be the unique element in the level $T_\alpha$ that is either $\pi(x)$ or $<_T$-below $\pi(x)$.
		We will show that $t^{\uparrow}\cap Y=\pi"(s^\uparrow \cap X)$.
		 
		We begin by showing that $\pi"(s^\uparrow \cap X)\subseteq t^{\uparrow}\cap Y$.
		Suppose otherwise and fix some $x'\in s^\uparrow \cap X$ such that $\pi(x')\notin t^{\uparrow}$, without loss of generality we assume that $x<_{lS} x'$.
		As $s<_S x'$, we get that $x'\notin N_\alpha$, so $\pi(x')\notin N_\alpha$ which implies $\h_T(\pi(x'))\geq \alpha$.
		Let $t'\in T_\alpha$ be the unique element equal to $\pi(x')$ or $<_T$-below $\pi(x')$.
		Since $t,t'\in T_\alpha$ and $t\neq t'$, we know that $\Delta_T(t, t')<\alpha$.
		As $x<_{lS} x'$ and $\pi$ is order-preserving, we get that $\pi(x) <_{lT} \pi(x')$, hence $t<_{lT} t'$.
		
		By Lemma~\ref{Lemma - Aron incom} there exists a $t''\in N_\alpha\cap T$ such that $t, t',t''$ are pairwise incomparable and $t<_{lT} t''<_{lT} t'$.
		Now by elementarity of $N_\alpha$, pick $y''\in Y\cap (t'')^{\uparrow }\cap N_\alpha$ and let $x'':=\pi^{-1}(y'')$, thus $\h_S(x'') <\alpha$.
		It follows that $t<_{lT} y'' <_{lT} t'$, which implies $y<_{lT} y'' <_{lT} y'$.
		Since $\pi^{-1}$ is order preserving, we get that $x<_{lS} x'' <_{lS} x'$.
		Note that this implies $x''\geq_S s$ which is an absurd.
		
		Next, we show that $t^{\uparrow}\cap Y \subseteq \pi"(s^\uparrow \cap X)$.
		Suppose otherwise and fix some $y\in t^\uparrow \cap Y$ such that $\pi^{-1}(y)\notin s^{\uparrow}$, without loss of generality we assume that $x<_{lS} \pi^{-1}(y)$.
		Let $s'$ be the unique element in the level $S_\alpha$ that is either $\pi^{-1}(y)$ or $<_S$-below $\pi^{-1}(y)$.
		Since $s,s'\in S_\alpha$ and $s\neq s'$, we know that $\Delta_S(s, s')<\alpha$.
		Thus $s<_{lS} s'$.
		As $\pi$ is order-preserving, we get that $\pi(x) <_{lT} y$,
		
		By Lemma~\ref{Lemma - Aron incom} there exists a $s''\in N_\alpha\cap S$ such that $s, s',s''$ are pairwise incomparable and $s<_{lS} s''<_{lS} s'$.
		Now by elementarity of $N_\alpha$, pick $x''\in X\cap (s'')^{\uparrow }\cap N_\alpha$, thus $\h_T(\pi(x'')) <\alpha$.
		It follows that $s<_{lS} x'' <_{lS} s'$, which implies $x<_{lS} x'' <_{lS} \pi^{-1}(y)$.
		As $\pi$ is order-preserving, we get that $\pi(x)<_{lT} \pi(x'') <_{lT} y$.
		Note that this implies $\pi(x'')\geq_T t$, hence we get a contradiction that $\h_T(\pi(x''))\geq \alpha$.

	\end{proof}
	
	\begin{claim}
		$f$ is a tree isomorphism and order isomorphism.
	\end{claim}
	\begin{proof}
		Clearly $f$ is injective.
		To check that $f$ is onto, fix some $t\in T\restriction C$, where $\alpha=\h_T(t)$.
		Pick some $y\in Y\cap t^{\uparrow}$.
		Then $y\notin N_\alpha$ and hence $\pi^{-1}(y)\notin N_\alpha$, so $\h_S(\pi^{-1}(y))\geq \alpha$.
		Let $s$ be the unique element in the level $S_\alpha$ which is $\pi^{-1}(y)$ or $<_S$-below $\pi^{-1}(y)$. 
		Using similar ideas as in the previous claim, one can show that $s^{\uparrow}\cap X=\pi^{-1}"(t^\uparrow \cap Y)$.
		This is suffices as $\pi$ is a bijection, so $\pi"(s^{\uparrow}\cap X)=t^\uparrow \cap Y$.
		As every two distinct $t,t'\in T_\alpha$, we will have $t^\uparrow \cap Y \neq t'^\uparrow \cap Y$, this implies $f(s)=t$ as sought.

		Next we show that $f$ is a tree isomorphism, pick arbitrary $s,s'\in S\restriction C$ such that $s<_S s'$.
		Pick any $x\in s'^{\uparrow}\cap X\subseteq s^\uparrow \cap X$, then $\pi(x)\in (f(s))^\uparrow \cap Y$ and $\pi(x)\in (f(s'))^\uparrow \cap Y$.
		So $f(s)<_T\pi(x)$ and $f(s')<_T \pi(x)$.
		As $f$ is level-preserving and $(\pi(x))_\downarrow$ is well-ordered, $f(s)<_T f(s')$.
		
		Finally, we show that $f$ is lexicographical order preserving.
		As $f$ is a tree isomorphism, we only need to verify the case $s<_{lS} s'$ and $s$ is $<_S$-incomparable with $s'$ for some $s,s'\in S\restriction C$.
		Clearly this implies that the elements in $s^\uparrow \cap X $ are $<_{lS}$-below the elements in the set $s'^\uparrow\cap X$.
		As $\pi$ is order-preserving, we get that the the elements of $f(s)^\uparrow \cap Y$ are $<_{lT}$-below the elements in the set $f(s')^\uparrow\cap Y$.
		So $f(s)<_{lT} f(s')$ as sought.\qedhere
\end{proof}	\qedhere	
\end{proof}

\begin{thm}\label{Proposition - many linear orders}
	Suppose $\p(\kappa,2,{\sq},\kappa)$ holds.
Then there exists a family $\langle \mathbf L_\zeta\mid\zeta<2^\kappa\rangle$ of linear orders of size $\kappa$ each one is minimal with respect to being non-$\sigma$-scattered and every two linear orders from the family are not embedded into each other.
\end{thm}

\begin{proof}
Our plan is to construct a family $\langle T^\zeta\mid \zeta\in {}^{\kappa}2\rangle$,
where for every $\zeta\in {}^{\kappa}2$, $T^\zeta \s{}^{<\kappa}\omega$ is a uniformly $\varrho$-coherent streamlined $\kappa$-Aronszajn not $\kappa$-Souslin tree,
and every subtree of $T^\zeta$
contains some frozen cone $T_{[s,i]}$ where $s\in T\restriction \nacc(\kappa)$ and $i<\omega$. 
Our construction of such a large family of trees follow the proof of \cite[Theorem~4.4]{paper62}.

Next, we will consider an antichain $X^{\zeta}\subseteq T^{\zeta}$ of size $\kappa$, so by Proposition~\ref{Proposition - the reduction} the linear order of size $\kappa$, $\mathbf L_\zeta:=(X^\zeta,{<_{\lex}})$ will be minimal with respect to being non-$\sigma$-scattered.

Finally, by using $\diamondsuit(H_\kappa)$ assumption in the construction of the family of trees and Lemma~\ref{Lemma - club iso}, we will establish that the family of linear orders $\langle X^{\zeta}\mid \zeta\in {}^{\kappa}2\rangle$ is indeed such that any two of them are not embedded to one another. 

Let $\langle C_\alpha\mid \alpha<\kappa\rangle$ be a sequence as in Lemma~\ref{Lemma - The wanted sequence},
without loss of generality, we may assume that $0\in C_\alpha$ for all nonzero $\alpha<\kappa$.

We shall construct a sequence $\langle L^\zeta\mid \zeta\in {}^{<\kappa}2\rangle$ such that, for all $\alpha<\kappa$ and $\zeta\in {}^{\alpha}2$:
\begin{enumerate}
	\item $L^\zeta\subseteq {}^{\alpha}\omega$;
	\item for every $\beta<\alpha$, $L^{\zeta\restriction \beta}=\{t\restriction \beta\mid t\in L^\zeta\}$.
\end{enumerate}
By convention, for every $\alpha\in \acc(\kappa+1)$ such that $\langle L^\zeta\mid \zeta\in {}^{<\alpha}2\rangle$ has already been defined, and for every $\zeta\in {}^{\alpha}2$, we shall let $T^\zeta:=\bigcup_{\beta<\alpha}L^{\zeta\restriction \beta}$, so that $T^\zeta$ is a tree of height $\alpha$ whose $\beta^{\text{th}}$ level is $L^{\zeta\restriction \beta}$ for all $\beta<\alpha$.

The construction of the sequence $\langle L^\zeta\mid \zeta\in {}^{<\kappa}2 \rangle$ is by recursion on $\dom(\zeta)$.
We start by letting $L^{\emptyset}:=\{\emptyset\}$.
For every $\alpha<\kappa$ such that $\langle L^\nu\mid \nu\in {}^{\alpha+1}2\rangle$ has already been defined, for every $\zeta\in {}^{\alpha+2}2$, let 
$L^\zeta:=\{ t{}^\smallfrown\langle n\rangle\mid t\in L^{\zeta\restriction\alpha+1}, n\in \omega\}$.

For $\alpha\in \acc(\kappa)$ such that $\langle L^\nu\mid \nu\in {}^{\alpha}2\rangle$ has already been defined, for every $\zeta\in {}^{\alpha+1}2$, and $t\in L^{\zeta\restriction\alpha}$, let $I_{t}:=\{n<\omega\mid \sup(t^{-1}(n))=\alpha\}$.
For every $\zeta\in {}^{\alpha+1}2$, let $L^{\zeta}:=\{ t{}^\smallfrown\langle \min(I_t)\rangle \mid t\in L^{\zeta\restriction \alpha}, \text{ and	 }I_t\neq \emptyset\}$.

Suppose now that $\alpha\in \acc(\kappa)$ is such that $\langle L^\zeta\mid \zeta\in {}^{<\alpha}2\rangle$ has already been defined.
We shall define a matrix
$$ \mathbb B^\alpha =\langle b^{\alpha,\zeta}\mid \beta\in C_\alpha,~\zeta\in {}^{\beta}2\rangle$$
ensuring that $b^{\alpha,\bar\zeta}\subseteq b^{\alpha,\zeta}\in L^\zeta$ whenever $\bar \zeta\subseteq \zeta$.
Then, for all $\zeta\in {}^{\alpha}2$, we will define $\mathbf b^\zeta :=\bigcup_{\beta\in C_\alpha}b^{\alpha,\zeta\restriction \beta}$ and we shall let 
$$L^\zeta:=\{ \eta*\mathbf b^\zeta\mid \eta\in\varrho, \dom(\eta)\s\alpha\}\setminus\{\emptyset\}.$$
	
We assure the following assumptions during the construction for every $\alpha\in \acc(\kappa)$ and $\zeta\in {}^{\alpha}2$:
\begin{enumerate}
	\item $C_\alpha =(\mathbf b^{\zeta})^{-1}(\{0\})$;
	\item if $\beta\in \acc(C_\alpha)$, then $\mathbf b^{\zeta\restriction \beta}\subseteq \mathbf b^\zeta$;
\end{enumerate}

Let $b^{\alpha,\emptyset}:=\emptyset$.
We now turn to define the components of the matrix $\mathbb B^\alpha$ by recursion on $\beta\in C_\alpha$.
So suppose that $\beta\in C_\alpha$ is such that
$$ \mathbb B^\alpha_{<\beta}:=\langle b^{\alpha,\zeta}\mid \bar \beta\in C_\alpha\cap \beta,~\zeta\in {}^{\bar \beta}2\rangle$$
has already been defined.

For all $\zeta\in {}^{\beta}2$ where $0<\beta$, there are  two main cases to consider:

	\underline{Case $(1)$:} Suppose that $\beta\in \nacc(C_\alpha)$ and denote $\beta^-:=\sup(C_\alpha\cap \beta)$.
	Next we split to three sub-cases:
	
\underline{Case $(1a)$:}  If $\beta\in \acc(\kappa)$ and the following holds:
\begin{enumerate}
	\item[$(i)$] There exists two distinct $\zeta_0,\zeta_1\in {}^{\beta}2$, $P\subseteq {}^{<\beta}\omega$ and a function $f:P\restriction E\rightarrow T^{\zeta_1}\restriction E$ where $P$ is a downward closed subset of $T^{\zeta_0}$, $E$ is a cofinal subset of $\beta$, $\acc(E)\cap \beta\subseteq E$ and $f$ is an injective order-preserving level-preserving map such that 
	$\Omega_\beta=\{(\langle \zeta_j\restriction \epsilon \mid j<2\rangle,f\restriction (P\restriction (E\cap \epsilon)) ) \mid \epsilon<\beta \}$.
	\item[$(ii)$] For some $i<2$, we have $\zeta=\zeta_i$.
	\item[$(iii)$] $\varphi(\beta)=(\eta_0,\eta_1)$ where $\eta_0,\eta_1\in \varrho$.
	\item[$(iv)$] for every $\gamma\in E$, the function $\eta_0*b^{\alpha,\zeta_0}\restriction \gamma$ is in $\dom(f)$.
\end{enumerate}
 
Define $b^{\alpha,\zeta_0}$ to be the $\lhd$-least element of $L^{\zeta_0}\cap T^{\zeta_0}_{[b^{\alpha,\zeta_0\restriction \beta^-}{}^\smallfrown\langle 0\rangle,1]}$.

Let $\gamma:= \min(E\setminus (\beta^-+1))$ and $\gamma^+:=\min(E\setminus (\gamma+1))$, i.e. $\gamma<\gamma^+$ are minimal two successive elements in $E$ above $\beta^-$.
Note that $(\eta_0*b^{\alpha,\zeta_0})\restriction\gamma^+\in \dom(f)$. 
Let $$n:=f((\eta_0*b^{\alpha,\zeta_0})\restriction\gamma^+)(\gamma)+\max\{|k|\mid k\in \rng(\eta_1)\}$$ and define $b^{\alpha,\zeta_1}$ to be the $\lhd$-least element of $L^{\zeta_1}\cap T^{\zeta_1}_{[b^{\alpha,\zeta_1\restriction \beta^-}{}^\smallfrown\langle 0\rangle,n+1]}$.
Notice that this set is non-empty as the tree $T^{\zeta_1}$ is $\varrho$-uniform.
Note this definition imply that $(\eta_1*b^{\alpha,\zeta_1})(\gamma)\neq f((\eta_0*b^{\alpha,\zeta_0})\restriction\gamma^+)(\gamma)$.

\underline{Case $(1b)$:} If $\beta\in \acc(\kappa)$ and the following holds:
\begin{enumerate}\item[$(i)$] There exists a subtree $A$ of $T^{\zeta}$ such that $\Omega_\beta=\{ (\zeta\restriction \epsilon, A\restriction \epsilon ) \mid  \epsilon<\beta \}$;
	\item[$(ii)$] $\varphi(\beta)$ is some $\eta\in \varrho$.
\end{enumerate}
Let
$$Q^{\alpha, \beta} := \{ t\in L^{\zeta}\cap T^{\zeta}_{[b^{\alpha,\zeta\restriction \beta^-}{}^\smallfrown\langle 0\rangle ,1]}\mid \exists s\in (T^{\zeta})\setminus A[  \varphi(\beta)*(b^{\alpha,\zeta\restriction \beta^-}{}^\smallfrown\langle 0\rangle)\s s~\wedge \varphi(\beta)^-*s\s t]\}.$$

Now, consider the two possibilities:
\begin{itemize}
	\item[$(1)$] If $Q^{\alpha,\beta} \neq \emptyset$, then let $t$ denote its $\lhd$-least element, and put $b^{\alpha,\zeta}:=t$;
	\item[$(2)$] Otherwise, let $b^{\alpha,\zeta}$ be the $\lhd$-least element of $L^{\zeta}\cap T^{\zeta}_{[b^{\alpha,\zeta\restriction \beta^-}{}^\smallfrown\langle 0\rangle,1]}$.
\end{itemize}

\underline{Case $(1c)$:} Otherwise, let $b^{\alpha,\zeta}$ be the $\lhd$-least element of $L^{\zeta}\cap T^{\zeta}_{[b^{\alpha,\zeta\restriction \beta^-}{}^\smallfrown\langle 0\rangle,1]}$.

\underline{Case $(2)$:} Suppose that $\beta\in \acc(C_\alpha)$.
Then we define $b^{\alpha,\zeta}:=\bigcup\{ b^{\alpha,\zeta\restriction \bar \beta}\mid \bar \beta\in C_\alpha\cap \beta \}$. We must show that the latter belongs to $L^\zeta$.
It suffices to prove that $b^{\alpha,\zeta}=\mathbf b^\eta$. Since $\vec C$ is coherent and $\beta\in \acc(C_\alpha)$ it is the case that $C_\alpha\cap \beta = C_\beta$,  and hence proving $b^{\alpha,\zeta}=\mathbf b ^\zeta$ amounts to showing that $b^{\alpha,\zeta\restriction \delta}=b^{\beta,\zeta\restriction \delta}$ for all $\delta\in C_\beta$. This is taken care of by the following claim.

\begin{claim}
	$\mathbb B^\alpha_{<\beta}=\mathbb B^\beta$. That is, the following matrices coincide:
	\begin{itemize}
		\item $\langle b^{\alpha,\zeta}\mid \bar\beta\in C_\alpha\cap\beta, \zeta\in {}^{\bar\beta}2 \rangle$;
		\item $\langle b^{\beta,\zeta}\mid \bar\beta\in C_\beta, \zeta\in {}^{\bar\beta}2\rangle$.
	\end{itemize}		
\end{claim}
\begin{proof} We already pointed out that $C_\alpha\cap\beta=C_{\beta}$, which for the scope of this proof we denote by $D$.
	Now, by induction on $\delta\in D$, we prove that 
	$$\langle b^{\alpha,\zeta}\mid \zeta\in {}^{\delta}2 \rangle=\langle b^{\beta,\zeta}\mid \zeta\in {}^{\delta}2 \rangle.$$
	
	The base case $\delta=\min(D)=0$ is immediate since $b^{\alpha,\emptyset}=\emptyset=b^{\beta,\emptyset}$.
	The limit case $\delta\in\acc(D)$ follows from the continuity of the matrices i.e. for $\gamma\in \acc(\kappa)$, $\delta\in\acc(C_\gamma)$ and $\zeta\in {}^{\delta}2$ it is the case that $b^{\gamma,\zeta}=\bigcup\{b^{\gamma,\zeta\restriction\bar\delta}\mid \bar\delta\in C_\gamma\cap\delta\}$.
	
	Finally, assuming that $\delta^-<\delta$ are two successive elements of $D$ such that 
	$$\langle b^{\alpha,\zeta}\mid \zeta\in {}^{{\delta^-}}2\rangle=\langle b^{\beta,\zeta}\mid \zeta\in {}^{{\delta^-}}2\rangle,$$
	we argue as follows.
	 Given $\zeta\in{}^\delta2$ by the above construction,
	for every $\gamma\in\{\alpha,\beta\}$, the value of 
	$b^{\gamma,\zeta}$ is completely determined by $\delta$, $\langle L^\zeta\mid\zeta\in {}^{\le\delta}2\rangle$, $\Omega_\delta$, $D$,
	$\psi(\delta)$, $\zeta$, $x$, and
	$\langle b^{\gamma,\xi}\mid \xi\in {}^{\delta^-}2\rangle$
	in such a way that our inductive assumptions imply that $b^{\alpha,\zeta}=b^{\beta,\zeta}$.
\end{proof}
At the end of the above process, for every $\zeta\in {}^{\kappa}2$, we have obtained a streamlined tree $T^\zeta:=\bigcup_{\alpha<\kappa}L^{\zeta\restriction \alpha}$ whose $\alpha^\text{th}$-level is $L^{\zeta\restriction \alpha}$.

As the construction is similar to previous one, we leave the details of the proof of the following Claim to the reader.

\begin{claim}
	For every $\zeta\in {}^{\kappa}2$, $T^\zeta$ is a slim uniformly $\varrho$-coherent
	 $\kappa$-Aronszajn streamlined subtree of ${}^{<\kappa}\omega$ that is not $\kappa$-Souslin tree and every subtree of $T$ contains some frozen cone $T_{[s,i]}$ where $s\in T\restriction \nacc(\kappa)$ and $i<\omega$.\qed
\end{claim}

For every $\zeta\in {}^\kappa2$ define $X^\zeta:=\{\mathbf  b^{\zeta\restriction \alpha}\mid \alpha<\kappa,~\otp(C_\alpha)=\omega\}$.
Recall that by Lemma~\ref{Lemma - The wanted sequence} $(3)$, the set $\{\alpha<\kappa\mid \otp(C_\alpha)=\omega\}$ is cofinal in $\kappa$.
By similar arguments as in Claim~\ref{Claim - not kappa souslin}, one can check that $X^\zeta$ is indeed an antichain of the tree $T^\zeta$ of size $\kappa$.
So by Proposition \ref{Proposition - the reduction}, 
the linear order $(X^\zeta,{<_{\lex}})$ is minimal with respect to being non-$\sigma$-scattered.
Using Lemma~\ref{Lemma - club iso} the next claim finish the proof.

\begin{claim}
	Suppose $\zeta_0,\zeta_1\in {}^{\kappa}2$ are distinct and $X$ is a subset of $T^{\zeta_0}$ of size $\kappa$, then there exists no club $E\subseteq \kappa$ and a map $f:X_\downarrow\restriction E\rightarrow T^{\zeta_1}\restriction E$ which is an order-preserving and level preserving injective map.
\end{claim}
\begin{proof}
	Assume on the contrary that for two distinct $\zeta_0,\zeta_1\in {}^{\kappa}2$ and $X\subseteq T^{\zeta_0}$ subset of size $\kappa$, there exists a club $E\subseteq \kappa$ and a map $f:(X)_\downarrow\restriction E\rightarrow T^{\zeta_1}\restriction E$ which is an order-preserving, level preserving injective map.
	We aim to get a contradiction.
	Set $\Omega:=\{ (\langle \zeta_j\restriction \epsilon \mid j<2\rangle,f\restriction (T^{\zeta_0}\restriction (E\cap \epsilon))\mid \epsilon<\kappa \}$, note that $\Omega\subseteq H_{\kappa}$.
	
	By a proof similar to that of \cite[Claim~2.3.2]{paper22}, for every $i<\kappa$, the set $$A_i := \{ \beta \in R_i \mid  \{ (\langle \zeta_j\restriction \epsilon \mid j<2\rangle,f\restriction (T^{\zeta_0}\restriction (E\cap \epsilon))\mid \epsilon<\beta \} = \Omega_\beta\text{ and }\sup(E\cap \beta)=\beta\}$$
		is stationary.
		Hence, the set $B_i:=A_i\cap \acc(\kappa)$ is a cofinal subset of $\acc(\kappa)$ for every $i<\kappa$.

	For every $\alpha<\kappa$, we let $m_\alpha$ be the set of all of $\varrho$-modifications with domain subset of $\alpha$.
	Thus, we apply clause~$(4)$ of Lemma \ref{Lemma - The wanted sequence} to the sequence $\langle B_i\mid i<\kappa\rangle$, and the club $D:=\{ \alpha<\kappa\mid m_\alpha\times m_\alpha\s\psi[\alpha]\}$
	to obtain
	an ordinal $\alpha\in D$ such that for all $i<\alpha$:
	\begin{equation}\label{eq1}
		\sup (\nacc(C_\alpha) \cap B_i) = \alpha.
	\end{equation}
	
	Note that $\alpha\in A$, to see that $f$ is not an order-preserving embedding, consider any $x \in (X)_\downarrow \cap T^{\zeta_0}_\alpha$.
	By the construction, $x$ and $f(x)$ are $\varrho$-modification of $  b^{\alpha,\zeta_0}$ and $b^{\alpha,\zeta_1}$ respectively, i.e. for some $\varrho$-modification $\eta_0$ and $\eta_1$ both with domain subset of $\alpha$ we have, $x=\eta_0* b^{\alpha,\zeta_0}$ and $f(x)=\eta_1* b^{\alpha,\zeta_1}$.
	As $\alpha\in D$ and $\eta_0,\eta_1 \in m_\alpha$, we can fix $i<\alpha$ such that $\psi(i)=(\eta_0,\eta_1)$.
	
	Fix $\beta \in \nacc(C_\alpha) \cap B_i$.
	Clearly, $\varphi(\beta)=\psi(\pi(\beta))=\psi(i)=(\eta_0,\eta_1)$.
	Since $\beta \in B_i$, $\sup(E\cap \beta)=\beta$ and $\{ (\langle \zeta_j\restriction \epsilon \mid j<2\rangle,f\restriction (T^{\zeta_0}\restriction (E\cap \epsilon))\mid \epsilon<\beta \} =\Omega_\beta$.
	
	Note that $b^{\alpha,\zeta_1}$ and $b^{\alpha,\zeta_0}$ are defined by Case $(1a)$ in the construction.

Let $\gamma<\gamma^+$ be the minimal two successive elements in $E$ above $\beta^-$.
Let $$n:=f(\eta_0*(b^{\alpha,\zeta_0}\restriction\gamma^+))(\gamma)+\max\{|k|\mid k\in \rng(\eta_1)\}$$
	By the way we constructed $b^{\alpha,\zeta_1}$ in Case $1a$, we know that $b^{\alpha,\zeta_1}(\gamma)>n$.
We get that $f(x)(\gamma)=(\eta_1*b^{\alpha,\zeta_1})(\gamma)>f(\eta_0*(b^{\alpha,\zeta_0}\restriction\gamma^+))(\gamma)$.
As $\eta_0*(b^{\alpha,\zeta_0}\restriction\gamma^+)\subseteq x$ and both are from $\dom(f)$, we get that as $f$ is order-preserving that $f(\eta_0*(b^{\alpha,\zeta_0}\restriction\gamma^+))\subseteq f(x)$.
Hence, $f(\eta_0*(b^{\alpha,\zeta_0}\restriction\gamma^+))(\gamma)=f(x)(\gamma)$ which give us the absurd statement that 
$f(x)(\gamma)<f(x)(\gamma)$.\qedhere
\end{proof}\qedhere
\end{proof}

For an infinite cardinal $\lambda$, recall that a $\lambda^+$-tree is \emph{special} if there is a partition of the tree to $\lambda$-many antichains.
Also recall that a linear order $(C,{<_C})$ is a \emph{$\lambda^+$-Countryman line} if $C$ is of cardinality $\lambda^+$ and $(C\times C,\leq^2_C)$ is a union of $\lambda$-many chains.
A summary of basic facts about $\lambda^+$-Countryman lines can be found in Section~$4$ of \cite{CEM23}.

\begin{cor} Suppose $\p_\xi(\kappa,2,{\sq},\kappa)$ holds with $\xi<\kappa$. 
Then there exists a family $\langle \mathbf L_\zeta\mid\zeta<2^\kappa\rangle$ of $\kappa$-Countryman lines each one is minimal with respect to being non-$\sigma$-scattered and every two members of the family are not embedded into each other.
\end{cor}
\begin{proof} 
	Under the assumptions, using the ideas of the proof in Claim~\ref{Claim - not kappa souslin}, the $\kappa$-trees constructed in Theorem~\ref{Proposition - many linear orders} are such that $T^\zeta\restriction \nacc(\kappa)$ is a special $\kappa$-tree whenever $\zeta\in {}^\kappa2$.
	So by the arguments presented in \cite{MR908147} and \cite[Proposition 4.3]{CEM23}, the linear orders are $\kappa$-Countryman lines as sought.
\end{proof}

It should be clear that all the constructions in this paper that are based on $\p_\xi(\kappa,2,{\sq},\kappa)$ can be carried out assuming the weaker principle $\p^\bullet_\xi(\kappa,2,{\sq},\kappa)$ of \cite{paper23}.
In particular, by \cite[Theorem~6.1(11)]{paper23}, after adding a single Cohen real to a model of $\ch$,
there exists a family of $2^{\aleph_1}$ many Countryman lines each one is minimal with respect to being non-$\sigma$-scattered and every two members of the family are not embedded into each other.

\section*{Acknowledgments}
This paper presents a result from the author’s PhD research at Bar-Ilan University under the
supervision of Assaf Rinot to whom he wishes to express his deep appreciation.
The author is supported by the European Research Council (grant agreement ERC-2018-StG 802756).
The main result of this paper was presented at the \emph{Set-Theoretic Topology} workshop in Oaxaca, August 2023. We thank the organizers for the invitation and the warm hospitality.
We thank Tanmay Inamdar and Shira Yadai for illuminating discussion on the subject.

\end{document}